\documentclass{amsart}
\usepackage{amsmath,amsfonts,amssymb}%
\usepackage[all]{xy}
\usepackage{setspace}
\usepackage{comment}
\usepackage[dvipsnames]{xcolor}
\usepackage{url}

\usepackage{tikz}

\usepackage{amsmath, amsthm, amsfonts, amssymb, stmaryrd}%
\usepackage{mathrsfs}

\usepackage{enumerate}

\usepackage{xparse, expl3}

\usepackage{url}
\usepackage{hyperref}
\hypersetup{
	colorlinks = true,
	citecolor=Blue,
	urlcolor=Blue,
	linkcolor=Green
}
\usepackage[capitalize]{cleveref}  
\crefname{lemma}{Lemma}{Lemmas}
\crefname{corollary}{Corollary}{Corollaries}
\crefname{theorem}{Theorem}{Theorems}
\crefname{equation}{Equation}{Equations}
\crefname{example}{Example}{Examples}
\crefname{section}{Section}{Sections}
\crefname{subsection}{Section}{Sections}

\newcommand\rest[1][]{\EM{\upharpoonleft_{#1}}}

\newcommand{\defn}[1]{{\bf{#1}}}

\def\Powerset{\mathcal{P}}
\def\PowersetFin{\Powerset_{<\w}}

\DeclareMathOperator{\ORD}{ORD}

\DeclareDocumentCommand{\given}{}{\EM{\, | \,}}

\DeclareMathOperator{\Rel}{Rel}
\DeclareMathOperator{\arity}{ar}

\DeclareDocumentCommand{\SeqSize}{d[] d()}
{
\EM{#1^{(#2)}}
}

\DeclareMathOperator{\SubStrop}{Sub}
\DeclareDocumentCommand{\SubStr}{d[] d()}
{
\IfNoValueTF{#1}
	{
        \EM{\SubStrop(#2)}
	}
	{
        \EM{\SubStrop_{#1}(#2)}
	}
}

\DeclareMathOperator{\SFop}{SF}
\DeclareDocumentCommand{\SF}{d[]}
{
\IfNoValueTF{#1}
	{
	\EM{\SFop}
	}
	{
	\EM{\SFop_{#1}}
	}
}

\DeclareDocumentCommand{\dual}{d()}
{
\IfNoValueTF{#1}
	{
	\EM{\hat{\ }}
	}
	{
	\EM{\hat{#1}}
	}
}

\DeclareDocumentCommand{\Rel}{d[]}
{
\IfNoValueTF{#1}
	{
	\EM{\mathcal{R}}
	}
	{
	\EM{\mathcal{R}_{#1}}
	}
}

\DeclareDocumentCommand{\Func}{d[]}
{
\IfNoValueTF{#1}
	{
	\EM{\mathcal{F}}
	}
	{
	\EM{\mathcal{F}_{#1}}
	}
}

\DeclareDocumentCommand{\ar}{d[]}
{
\IfNoValueTF{#1}
	{
	\EM{\arity}
	}
	{
	\EM{\arity_{#1}}
	}
}

\DeclareDocumentCommand{\Fn}{d<> d[] d()}
{
\IfNoValueTF{#3}
	{
	\EM{\textrm{Fn}(#1, #2)}
	}
	{
	\EM{\textrm{Fn}(#1, #2, #3)}
	}
}

\DeclareDocumentCommand{\AllColors}{d<> d[] d()}
{
\IfNoValueTF{#3}
	{
	\EM{\mathfrak{AC}_{#1}^{#2}}
	}
	{
	\EM{\mathfrak{AC}_{#1}^{#2}(#3)}
	}
}

\DeclareDocumentCommand{\AllSetsColors}{d<> d[] d()}
{
\IfNoValueTF{#3}
	{
	\EM{\mathfrak{ASC}_{#1}^{#2}}
	}
	{
	\EM{\mathfrak{ASC}_{#1}^{#2}(#3)}
	}
}

\DeclareDocumentCommand{\GoodEdges}{d<> d[] d()}
{
\IfNoValueTF{#3}
	{
	\EM{\mathfrak{GE}_{#1}^{#2}}
	}
	{
	\EM{\mathfrak{GE}_{#1}^{#2}(#3)}
	}
}

\DeclareDocumentCommand{\Middle}{d~~ d<> d[] d()}
{
\IfNoValueTF{#4}
	{
	\EM{\textrm{Mid}_{#3}}
	}
	{
	\EM{\textrm{Mid}_{#3}(#4)}
	}
}

\DeclareMathOperator{\Sym}{Sym}
\DeclareDocumentCommand{\Perm}{d()}
{
\EM{\Sym(#1)}
}

\DeclareDocumentCommand{\Pushforward}{d[] d()}
{
\IfNoValueTF{#2}
{
    \EM{{#1}^*}
}
{
    \EM{{#1}^*(#2)}
}
}

\DeclareDocumentCommand{\Neigh}{d<> d[] d()}
{
    \EM{\approx^{#1}_{#2:#3}}
}

\newcommand{\rainbowsymbol}{%
    \begin{tikzpicture}[scale=0.3]
        \draw[line width=1.5pt, black] (0,-0.2) arc (0:180:0.4);
    \end{tikzpicture}%
}

\DeclareMathOperator{\graphop}{Gra}

\DeclareMathOperator{\Rainbowop}{\rainbowsymbol}

\DeclareMathOperator{\RainbowThresholdop}{\rainbowsymbol}
\DeclareDocumentCommand{\RainSeqSupScript}{d[] d()}
{
\IfNoValueTF{#2}
{
    \EM{\textrm{Seq}^{\!\!\!\!\!\Rainbowop}_{#1}}
}
{
    \EM{\textrm{Seq}^{\!\!\!\!\!\Rainbowop}_{#1}(#2)}
}
}

\DeclareDocumentCommand{\RainFont}{m}{\mathbb{#1}}
\DeclareDocumentCommand{\RainSeq}{d[] d()}
{
\IfNoValueTF{#2}
{
    \EM{\RainFont{S}_{#1}}
}
{
    \EM{\RainFont{S}_{#1}(#2)}
}
}

\DeclareDocumentCommand{\RainGraphSupScript}{d[] d()}
{
\IfNoValueTF{#2}
{
    \EM{\textrm{G}^{\!\!\!\!\!\!\RainbowThresholdop}_{#1}}
}
{
    \EM{\textrm{G}^{\!\!\!\!\!\!\RainbowThresholdop}_{#1}(#2)}
}
}

\DeclareDocumentCommand{\RainGraph}{d[] d()}
{
\IfNoValueTF{#2}
{
    \EM{\RainFont{G}_{#1}}
}
{
    \EM{\RainFont{G}_{#1}(#2)}
}
}

\DeclareDocumentCommand{\SeqToGraph}{}
{
    \EM{\iota^{\!\!\!\!\Rainbowop}}
}

\DeclareDocumentCommand{\Graphs}{d()}
{
\IfNoValueTF{#1}
{
    \EM{\graphop}
}
{
    \EM{\graphop(#1)}
}
}

\DeclareDocumentCommand{\floor}{m}
{
\lfloor #1 \rfloor
}

\DeclareDocumentCommand{\bbfloor}{m}
{
\bigg\lfloor #1 \bigg\rfloor
}

\makeatletter
\newcommand{\dotminus}{\mathbin{\text{\@dotminus}}}

\newcommand{\@dotminus}{%
  \ooalign{\hidewidth\raise1ex\hbox{.}\hidewidth\cr$\m@th-$\cr}%
}

\DeclareDocumentCommand{\quadiff}{}
{
    \EM{\quad\text{ if and only if }\quad}
}

\DeclareDocumentCommand{\SubStr}{d[] d()}
{
\IfNoValueTF{#1}
	{
        \EM{\SubStrop(#2)}
	}
	{
        \EM{\SubStrop_{#1}(#2)}
	}
}

\def\cG{{\EM{\mathcal{G}}}}
\def\cH{{\EM{\mathcal{H}}}}

\def\cU{{\EM{\mathcal{U}}}}

\def\cX{{\EM{\mathcal{X}}}}

\def\w{\EM{\omega}}

\def\^{\EM{{}^{\And}}}

\def\And{\EM{\wedge}}

\def\<{\EM{\langle}}
\def\>{\EM{\rangle}}

\def\EM#1{\ensuremath{#1}}

\def\st{\,:\,}
\def\:{\colon}

\providecommand{\dotdiv}{
  \mathbin{
    \vphantom{+}
    \text{
      \mathsurround=0pt 
      \ooalign{
        \noalign{\kern-.35ex}
        \hidewidth$\smash{\cdot}$\hidewidth\cr 
        \noalign{\kern.35ex}
        $-$\cr 
      }%
    }%
  }%
}

\DeclareDocumentCommand{\RightJustify}{m}{\hspace*{\fill}\mbox{#1}\penalty-9999\relax}

\newcounter{margincounter}
\DeclareDocumentCommand{\displaycounter}{}
	{{\arabic{margincounter}}}
\DeclareDocumentCommand{\incdisplaycounter}{}
	{{\stepcounter{margincounter}\arabic{margincounter}}}

\DeclareDocumentCommand{\DeclareComment}{m m m o d()}{%
\expandafter\DeclareDocumentCommand\csname Hide#1\endcsname {}
	{%
	\expandafter\DeclareDocumentCommand\csname #1\endcsname {+m} {}
	
	\expandafter\DeclareDocumentCommand\csname f#1\endcsname {+m} {}

	\expandafter\DeclareDocumentEnvironment{e#1} {} {} {}
	}

\expandafter\DeclareDocumentCommand\csname Show#1\endcsname {}
	{
	\expandafter\DeclareDocumentCommand\csname #1\endcsname {+m}
		{%
		\textcolor{#2}
			{ 
			{\tiny \textbf{(#3)}}
			\IfValueT{#5}
				{%
				#5
				}
			####1
			}
		}

	\expandafter\DeclareDocumentCommand\csname f#1\endcsname {+m}
		{%
		\IfValueTF{#4}
			{
			\textcolor{#2}
			{\text{$\,^{(\incdisplaycounter{#4})}$}}
			\marginpar{\tiny\textcolor{#2}{
				{\text{\tiny $(\displaycounter{#4})$}}
				\text{\IfValueT{#5}{#5}
				####1}}}
			}
			{
			\textcolor{#2}
			{$\,^{(\incdisplaycounter)}$}
			\marginpar{\tiny\textcolor{#2}{
				{\tiny $(\displaycounter)$}
				\text{\IfValueT{#5}{#5}
				####1}}}
			}
		}

	\expandafter\DeclareDocumentEnvironment{e#1} {}
		{
		\textcolor{#2}
		\bgroup
		\IfValueT{#5}
			{%
			#5
			}
		}
		{
		\egroup
		}
	}

\csname Show#1\endcsname

}

\definecolor{NAColor}{rgb}{1.0,0.0,0.0}
\definecolor{ProblemColor}{rgb}{0.7,0.1,0.7}
\definecolor{TBDColor}{rgb}{0.0,0.0,0.8}
\definecolor{MathColor}{rgb}{0.0,0.4,0.1}
\definecolor{NateColor}{rgb}{0.0,0.5,1.0}
\definecolor{MostafaColor}{rgb}{1.0,0.0,1.0}
\definecolor{RefColor}{rgb}{1.0,0.0,1.0}
\definecolor{LaterColor}{rgb}{1.0,0.0,1.0}

\DeclareComment{NA}{NAColor}{\EM{\star}}
\DeclareComment{TBD}{TBDColor}{\EM{!}}
\DeclareComment{PROBLEM}{ProblemColor}{\EM{!!}}(\textbf)
\DeclareComment{MATH}{MathColor}{\EM{\#}}
\DeclareComment{NATE}{NateColor}{\EM{\square}}(N:)
\DeclareComment{MOSTAFA}{MostafaColor}{\EM{\square}}(M:)
\DeclareComment{REF}{RefColor}{\EM{(r)}}(Ref:)
\DeclareComment{LATER}{LaterColor}{\EM{(L)}}(Later:)

\DeclareDocumentCommand{\DeclareCounter}{m}%
{%
 	\ifcsname c@#1\endcsname%
	\else%
		\newcounter{#1}%
	\fi%
}

\DeclareDocumentCommand{\MyQED}{}{\qed}

\DeclareDocumentEnvironment{proof}{d() D[]{\MyQED} d<>}
{
\noindent\IfNoValueTF{#1}
{\emph{Proof.\!\!}}
{\emph{Proof\ #1.\ }}
}
{
\IfValueTF{#3}
	{%
	\ExplSyntaxOff
	\RightJustify{\EM{#2_{#3}}}
	\vspace{1ex}
	}
	{
	\RightJustify{\EM{#2}}
	\vspace{1ex}
	}
}

\DeclareCounter{ProofLabelcOUntEr}
\DeclareDocumentCommand{\ProofLabel}{}{%
\addtocounter{ProofLabelcOUntEr}{1}
\label{cUrrEntProoflAbEl\arabic{ProofLabelcOUntEr}}
}

\DeclareDocumentCommand{\ProofRef}{D<>{1}}
{%
\ref{cUrrEntProoflAbEl\arabic{ProofcOUntEr#1}}
}
\DeclareDocumentCommand{\ProofCref}{D<>{1}}
{%
\cref{cUrrEntProoflAbEl\arabic{ProofcOUntEr#1}}
}

\DeclareDocumentEnvironment{Proof}{d() D[]{\MyQED} D<>{1}}
{
\DeclareCounter{ProofcOUntEr#3}%
\setcounter{ProofcOUntEr#3}{\value{ProofLabelcOUntEr}}%
\begin{proof}(#1)[#2]<\ProofRef<#3>>%
}
{
\end{proof}
}

\ExplSyntaxOn
\def\TheoremDepth{section}

\DeclareDocumentCommand{\DeclareTheorem}{m o m o}{%
\IfNoValueTF{#4}
	{%
	\IfNoValueTF{#2}
		{%
		\newtheorem{#1vArIAblE}{#3}
		}
		{%
		\newtheorem{#1vArIAblE}[#2vArIAblE]{#3}
		}
	}
	{%
	\newtheorem{#1vArIAblE}{#3}[#4]%
	}
\newtheorem*{#1vArIAblE*}{#3}

\DeclareDocumentEnvironment{#1}{o o}

	{
	\IfValueT{##2}%
		{
		\begin{spacing}{##2}
		}
	\IfValueTF{##1}
		{
		\begin{#1vArIAblE}[##1]
		}
		{
		\begin{#1vArIAblE}
		}
	\ProofLabel
	}
	{
	\IfValueT{##2}%
		{
		\end{spacing}{##2}
		}
	\end{#1vArIAblE}
	}

\DeclareDocumentEnvironment{#1*}{o o}

	{
	\IfValueT{##2}%
		{
		\begin{spacing}{##2}
		}
	\IfValueTF{##1}
		{
		\begin{#1vArIAblE*}[##1]
		}
		{
		\begin{#1vArIAblE*}
		}
	}
	{
	\IfValueT{##2}%
		{
		\end{spacing}{##2}
		}
	\end{#1vArIAblE*}
	}
}
\ExplSyntaxOff

\theoremstyle{plain}
\DeclareTheorem{theorem}{Theorem}[\TheoremDepth]
\DeclareTheorem{proposition}[theorem]{Proposition}
\DeclareTheorem{lemma}[theorem]{Lemma}
\DeclareTheorem{corollary}[theorem]{Corollary}
\DeclareTheorem{claim}[theorem]{Claim}

\theoremstyle{definition}
\DeclareTheorem{definition}[theorem]{Definition}
\DeclareTheorem{axiom}[theorem]{Axiom}
\DeclareTheorem{convention}[theorem]{Convention}

\theoremstyle{remark}
\DeclareTheorem{remark}[theorem]{Remark}
\DeclareTheorem{example}[theorem]{Example}
\DeclareTheorem{question}[theorem]{Question}
\DeclareTheorem{conjecture}[theorem]{Conjecture}

\begin{document}

\title{Rainbow Threshold Graphs}

\begin{abstract}
We define a generalization of \emph{threshold graphs} which we call $k$-rainbow threshold graphs. We show that the collection of $k$-rainbow threshold graphs do not satisfy the $0$-$1$ law for first order logic and that asymptotically almost surely all $(k+1)$-rainbow threshold graphs are not isomorphic to a $k$-rainbow threshold graph. 
\end{abstract}

\author{Nathanael Ackerman}
\address{Harvard University,
Cambridge, MA 02138, USA}
\email{nate@aleph0.net}

\author{Mostafa Mirabi}
\address{The Taft School, Watertown, CT 06795, USA}
\email{mmirabi@wesleyan.edu}

%
%


\maketitle

\section{Introduction}

A threshold graph is one which can be obtained by starting with an isolated vertex and in sequence adding either a new isolated vertex or a vertex which dominates all vertexes already constructed. The collection of threshold graphs were introduced by Chv\`{a}tal and Hammer (\cite{MR479384}) and Henderson and Zalcstein (\cite{MR488948}) in 1977. Due to the numerous different representations of threshold graphs and their simplicity of construction, threshold graphs have found applications in a wide range of areas of computer science and graph theory. See, for example, \cite{MR1417258} for more details. 

When considering a graph process whereby vertexes are added one at a time, we can imagine an agent which at each stage is handed a new vertex and the goes through all the previous vertices to decide whether or to connect it to the new vertex. In the case of threshold graphs  this agent is unable to distinguish between previous vertices when making his decision, and hence has to make the same decision for each previous vertex. 


A natural generalization of this construction is where the agent is able to mark each vertex it is given with one of $k$-colors. Then when future vertices are added it is able to use the color of the previous vertices in make its decision on whether or not to connect to the new vertex (i.e. each new vertex treats all previous vertices of the same color the same). This construction gives rise to a \emph{$k$-rainbow threshold graph}.


In this paper we will show that the collection of $k$-rainbow threshold graphs, for different values of $k$, provide a stratification of all graphs. We will also show a that asymptotically almost surely a $k+1$-rainbow threshold graph is not isomorphic to a $k$-rainbow threshold graph. This tells us that by allowing our agents access to more colors there are new graphs which can be constructed. We will also show that for no $k$ does the collection of $k$-rainbow threshold graphs satisfy the $0$-$1$ law for first order logic.

\subsection{Related Work}

Due to the numerous equivalent definitions of threshold graphs, there have been many different generalizations of threshold graphs. For a survey of some of these see \cite{MR1686154} Chapter 14. 

While we show that the collection of $k$-rainbow threshold graphs (and hence the collection of threshold graphs) do not satisfy the $0$-$1$ law for first order logic, we can still ask about other limiting notions. In particular the limits of threshold graphs, in the sense of graphons, is studied in \cite{MR2573956}.

\subsection{Notation}

We let $\ORD$ be the collection of ordinals. For $k \in \w$ we let $[k] = \{0, 1, \dots, k-1\}$. We let $[k, n) = \{k, k+1, \dots, n-1\}$. When no confusion can arise we will also use $k$ for $[k]$. If $X$ is a set we let $\Powerset(X)$ denote the collection of subsets of $X$ and $\PowersetFin(X)$ the collection of finite subsets of $X$.  


\section{Rainbow Threshold Graphs}

We now introduce the notion of a $k$-rainbow threshold graph and give some of their basic properties. 

\begin{definition}
Suppose $K, X$ are sets. By a \defn{$K$-rainbow sequence} on $X$ we mean a pair $(a,e)$ where both $a$ and $e$ have domain $X$, $a$ has codomain $K$ and $e$ has codomain $\Powerset(K)$. We call $a$ the \defn{colors} and $e$ the \defn{sets of colors}. 

We let $\RainSeq[K](X)$ be the collection of $K$-rainbow sequences on $X$. For $k, n \in \w$ we let $\RainSeq[k](n) = \RainSeq[{[k]}]({[n]})$. 

For $(a_0, e_0), (a_1, e_1) \in \RainSeq[k](X)$ we say $(a_0, e_0)$ is \defn{similar} to $(a_1, e_1)$, denoted $(a_0, e_0) \sim (a_0, e_0)$ if there is a permutation of $\tau$ of $[k]$ where for all $x \in X$ we have $a_0(x) = \tau(a_1(x))$ and $e_0(x) = \tau``[e_1(x)]$. 
\end{definition}

Intuitively two $K$-rainbow sequences are similar if we can obtain one from the other by \emph{relabeling} the colors. 

\begin{definition}
Suppose $K$ is a set and $\cX = (X, \leq)$ is a linear ordering and $S = (a, e)$ is a $k$-rainbow sequence on $X$. 

Let $\cG_{S, \leq} = (X, E_S)$ be the graph where 
\[
(\forall x < y\in X)\,(x, y) \in E_S\quadiff a(x) \in e(y).
\]
We say $\cG_{S, \leq}$ is the \defn{$\cX$-rainbow threshold graph} associated with $S$ and $\cX$. We say a graph is a \defn{$k$-rainbow threshold graph} (on $\cX$) if it is the $k$-rainbow graph associated with some $K$-rainbow threshold sequence $S$ and $\cX$.

If $A \subseteq \w$ is finite and $S \in \RainSeq[{[k]}](A)$ then we let $\SeqToGraph(S)$ be the $(A, \leq)$-rainbow threshold graph associated with $S$. 



For $A \in \PowersetFin(\w)$ we let $\RainGraph[k](A) = \SeqToGraph``[\RainSeq[k](A)]$ we say elements of $\RainGraph[k](A)$ are \defn{on $A$}.  We also let $\RainSeq[k]= \bigcup_{A \in \PowersetFin(\w)} \RainSeq[k](A)$. 

\end{definition}

We now give several examples of $k$-rainbow threshold graphs.

\begin{example}
A \defn{threshold graph} on $[n]$ is a graph $\cG$ such that there is a sequence $\<e_i\>_{i \in [n]}$ where $(\forall i \in [n])\, e_i \in \{0, 1\}$ and $(i, j) \in \cG$ if and only if either $i < j$ and $e_j = 1$ or $j < i$ and $e_i = 1$.

Suppose $\<e_i\>_{i \in [n]}$ is a sequence of elemets of $\{0, 1\}$. Let $e^*\:[n] \to \Powerset([1])$ be the map where $e^*(i) = \emptyset$ if and only $e_i = 0$. Let $a^*\:[n] \to [1]$ be the unique such function.

If $\cG$ is a threshold graph on $[n]$ with sequence $\<e_i\>_{i \in [n]}$ then $\SeqToGraph({(a^*, e^*}) = \cG$.  Hence the collection of $1$-rainbow threshold graphs is the same as the collection of threshold graphs. In particular this motivates the use of the term \emph{rainbow threshold graph}.
\end{example}

\begin{example}
A \defn{threshold bigraph} is a bipartite graph $\cH = (H, E)$ where the following hold. 
\begin{itemize}
\item $H = X \cup Y$ where $X$ and $Y$ are the two parts of the bipartite graph. 

\item There is a linear ordering $\leq$ on $X \cup Y$. 

\item There is a function $\alpha\:X \cup Y \to \{0, 1\}$. 

\item For all $x_0, x_1 \in X$ and $(y_0, y_1) \in Y$, $(x_0, x_1), (y_0, y_1) \not \in E$. 

\item For all $z_0,z_1 \in X \cup Y$ with $x_0 < z_1$ and 
\[
z_0 \in X\quadiff z_1 \in Y
\]
we have $(x, y) \in E$ if and only if $\alpha(z_1) = 1$. 
\end{itemize}

Note the graph $\cH$ above is a threshold bigraph is also a $2$-rainbow threshold graph. Specifically let $S =(a, e)$ be the $2$-rainbow sequence where the following hold. 
\begin{itemize}
\item $(\forall x \in H)\, a(x) = 0$ if and only if $x \in X$. 

\item $(\forall x \in H)\, a(x) \not \in e(x)$. 

\item $(\forall x \in H)\, 1-a(x) \in e(x) \leftrightarrow \alpha(x) = 1$. 

\end{itemize}
We then have $\cG_{S, \leq}$, i.e. the threshold graph corresponding to $S$ and $(H, \leq)$ is equal to $\cH$. 
\end{example}

Note unlike with threshold graphs, for $k$-rainbow threshold graphs (with $k > 1$) there are in general multiple $k$-rainbow threshold sequences which give rise to the same graph. In particular the following is immediate. 

\begin{lemma}
\label{Simiarity of rainbow sequences leads to identical graphs}
Suppose $A \subseteq \w$ and $S_0, S_1 \in \RainSeq[k](A)$ with $S_0 \sim S_1$. Then $\SeqToGraph(S_0) = \SeqToGraph(S_1)$.  
\end{lemma}

\cref{Simiarity of rainbow sequences leads to identical graphs} is not the only cause for redundancy in the construction of $k$-rainbow threshold graphs. But, as $\RainSeq[k](A)$ is much easier to study than $\RainGraph[k](A)$, understanding these redundancies, and hence the structure of the map $\SeqToGraph$ on $A$ will be important.

Note that $\RainGraph[k]$ is closed under subgraphs. The following is immediate. 
\begin{lemma}
If $S = (a, e) \in \RainSeq[k](X)$ and $X_0 \subseteq X$ then the following hold. 
\begin{itemize}
\item $S \rest[X_0] \in \RainSeq[k](X)$. 

\item For any linear ordering $\leq$ on $X$ with $\leq_0$ the restriction to $X_0$, $\cG_{S_0, \leq_0} = \cG_{S, \leq}\rest[X_0]$. 
\end{itemize}
\end{lemma}

We also have that every graph is a $K$-rainbow threshold graph for some set $K$. 
\begin{proposition}
For every linear ordering $\cX = (X, \leq)$ every graph on $X$ is an $\cX$-rainbow threshold graph. 
\end{proposition}
\begin{proof}
Suppose $\cG = (X, E)$. For $i \in X$ let $e\:X \to \Powerset(X)$ where $(\forall x \in X)\, e(x) = \{j < i \st (i, j) \in E\}$. Let $a\:X \to X$ be the identity. Then $\cG$ is the $\cX$-rainbow threshold graph associated with $(a, e)$.  
\end{proof}

The following is lemma is immediate. 
\begin{lemma}
If $k \subseteq \ell$ then every $k$-rainbow threshold graph (on $\cX$) is an $\ell$-threshold graph (on $\cX$). 
\end{lemma}

In particular for any cardinal $\kappa$, the sequence $\<\RainGraph[k](\kappa, \in)\>_{k \in \ORD}$ provide a stratification of the collection of all graphs on $\kappa$.

%



\section{Minimal Colors}

In this section we will show that asymptotically almost surely a uniformly selected $k+1$-rainbow threshold graph is not isomorphic to any $k$-rainbow threshold graph. We will do this by considering the ``neighborhood'' equivalence relation relative to a set $X$. Specifically $i, j \not \in X$ have the same $X$-neighborhood if $(\forall x \in X)\, E(i, x) \leftrightarrow E(j, x)$. We will provide an upper bound, in terms of the size of $X$,  on the number of equivalence classes an $X$-neighborhood relation can have in a $k$-rainbow threshold graph. We will then show that in ``most'' $k+1$-rainbow threshold graphs there is an $X$ such that the number of equivalence classes of the $X$-neighborhood relation is larger then this bound.

\begin{definition}
If $\cG$ is a graph on $G$ and $A, X \subseteq G$.  Let $\Neigh<\cG>[A](X)$ be the equivalence relation on $A$ where for $i, j \in A$ we have $i \Neigh<\cG>[A](X) j$ if and only if 
\begin{itemize}
\item $i = j \in X$, or 

\item $i, j \not \in X$ and $(\forall c \in X)\, E(i, c) \leftrightarrow E(j, c)$. 
\end{itemize}
For $S \in \RainSeq[k](n)$ we let $\Neigh<S>[A](X) = \Neigh<\SeqToGraph(S)>[A](X)$. 
\end{definition}

\begin{definition}
Suppose $X \subseteq[n]$ and $i, j \in [n]$. We say $i \equiv_X j$ if and only if 
\begin{itemize}
\item $i = j \in X$, or 

\item $i, j \not \in X$ and $(\forall c \in X)\, i < x \leftrightarrow j < x$. 
\end{itemize}
\end{definition}

\begin{proposition}
\label{Upper bound on neighborhood equivalence classes}
Suppose $\cG \in \RainGraph[k](n)$ and $X \subseteq [n]$. Then $\Neigh<\cG>[{[n]\setminus X}](X)$ has at most $k \cdot 2^k \cdot (1 + \frac{X}{2})$-many equivalence classes. 
\end{proposition}
\begin{proof}
Suppose $\cG = \SeqToGraph(S)$ for $S = (a, e) \in \RainSeq[k](n)$.  Fix a set $X \subseteq [n]$ of size $\ell$ and let $(x_i)_{i \in [\ell]}$ be an enumeration of the elements of $X$. For $s \in [k] \times \Powerset([k])$ let $Q_{s} = \{c \in [n] \setminus X \st S(c) = s\}$.  For $Y \subseteq X$ let $r_{s, Y}$ be the number of $\Neigh<\cG>[Q_{s}](Y)$-equivalence classes. Let $t_Y = \sum_{s \in [k] \times \Powerset([k])} r_{s, Y}$

For any $i,  j \in [n] \setminus X$, if $i \equiv_Y j$ and $S(i) = S(j)$ then $i \Neigh<\cG>[{[n] \setminus X}](Y) j$. Therefore there are at most $t_Y$-many $\Neigh<\cG>[{[n] \setminus X}](Y)$-equivalence classes. 

For $i \in [\ell+1]$ let $X_i = \{x_{j}\}_{j \in [i]}$. Note $X_0$ is empty and so $r_{s, X_0} \leq 1$ for all $s \in [k] \times \Powerset([k])$ and $t_{X_0} \leq k \cdot 2^k$. 

Let  $s = (a_s, e_s) \in [k] \times \Powerset([k])$. If $a(x_i) \in e_s \leftrightarrow a_s \in e(x_i)$ then $\Neigh<\cG>[Q_{s}](X_{i+1})$ is the same as $\Neigh<\cG>[Q_{s}](X_{i})$ and hence $r_{s, X_i} = r_{s, X_{i+1}}$. 

But if $a(x_i) \in e_s \leftrightarrow a_s \not \in e(x_i)$ then there is a single $\Neigh<\cG>[Q_{s}](X_{i})$-neighborhood which may consists of two $\Neigh<\cG>[Q_{s}](X_{i+1})$-neighborhoods, and all other $\Neigh<\cG>[Q_{s}](X_{i})$-neighborhoods are also $\Neigh<\cG>[Q_{s}](X_{i+1})$-neighborhoods. Therefore $r_{s, X_{i+1}} \leq r_{s, X_i} +1$. 

However for any $x_i$, $|\{(a_s, e_s) \st a_s \in e(x_i) \leftrightarrow a(x_i) \not \in e_s\}| = \frac{k \cdot 2^k}{2}$. Therefore $t_{X_{i+1}} \leq t_{X_i} + \frac{k \cdot 2^k}{2}$. But then 
\[
t_{X_\ell} \leq k \cdot 2^k + k \cdot 2^k \cdot {\displaystyle\frac{X}{2}} = k \cdot 2^k \cdot \bigg(1 + {\displaystyle\frac{X}{2}}\bigg).
\]
as desired. 
\end{proof}

We now show that for most $\cG \in \RainGraph[k+1](n)$ there is a subset $X \subseteq [n]$ of sufficient size such that the number of $\Neigh<\cG>[{[n]\setminus X}](X)$-equivalence classes is larger than the upper bound given in \cref{Upper bound on neighborhood equivalence classes} for $k$-rainbow threshold graphs. 

First though we need to give conditions which will ensure that we can recover (up to relabeling) the colors and sets of colors of \emph{most} elements in a $k$-rainbow threshold graph.

\begin{definition}
Suppose $A \subseteq X$. Let 
\[
\AllColors<k>[X](A) = \{(a, e) \in \RainSeq[k](X) \st (\forall i \in [k])(\exists x \in A)\, a(x) = i\}.
\]
If $S \in \AllColors<k>[X](A)$ we say $S$ \defn{has all colors on $A$}. We say a graph $\cG$ with underlying set $X$ \defn{has all colors on $A$} if it is in $\SeqToGraph``[\AllColors<k>[X](A)]$. 
\end{definition}

\begin{definition}
Suppose $A \subseteq X$. Let 
\[
\AllSetsColors<k>[X](A) = \{(a, e) \in \RainSeq[k](X) \st (\forall i \in [k])(\exists x, y \in A)\, (i \in e(x)) \leftrightarrow (i \not \in e(y))\}.
\]
If $S \in \AllSetsColors<k>[X](A)$ we say $S$ \defn{separates all colors on $A$}. We say an graph $\cG$ with underlying set $X$ \defn{separates all colors on $A$} if it is in $\SeqToGraph``[\AllSetsColors<k>[X](A)]$. 
\end{definition}

The following two lemmas are immediate but very important for recovering information about the $k$-rainbow threshold sequence used to construct a $k$-threshold graph. 

\begin{lemma}
\label{All colors on the left imply you can determine equivalence of subsets on the right}
Suppose $S = (a, e) \in \AllColors<k>[n](X)$, $i, j \in [n]$ and $(\forall c \in X)\, c \leq i, j$. Then $i \Neigh<S>[n](X) j$ if and only if $e(i) = e(j)$. 
\end{lemma}

\begin{lemma}
\label{All sets of colors on the right imply you can determine equivalence of colors on the left}
Suppose $S = (a, e)\in \AllSetsColors<k>[n](X)$ and $(\forall c \in A)\, c \geq i, j$. Then $i \Neigh<S>[n](X) j$ if and only if $a(i) = a(j)$. 
\end{lemma}

We now define a collection of $k$-rainbow threshold graphs which for which there is a set $X$ where the $X$-neighborhood relation has close to the maximal possible equivalence classes. 

\begin{definition}
Suppose $k, \ell, n \in \w$. We say a sequence $S \in \RainSeq[k](n)$ is \defn{$\ell$-good} (for $k$) if for every 
\[
\bigg(\forall r \leq \bbfloor{\frac{n}{\ell}} - 3\bigg)(\forall s \in [k] \times \Powerset([k]))\, S^{-1}(s) \cap \big[(r+1) \cdot \ell, (r+2) \cdot \ell\big) \neq \emptyset.
\]

We say a graph $\cG$ is \defn{$\ell$-good} (for $k$) if there is an sequence $S \in \RainSeq[k](n)$ which is $\ell$-good (for $k$) with $\cG = \SeqToGraph(S)$. 
\end{definition}

\begin{lemma}
\label{Bounds on ell-good graphs}
Suppose $k, n, \ell \in \w$ with $\ell \geq 2^k$. Let 
\[
\delta_{k, n}(\ell) = \bbfloor{\frac{n}{\ell}}\cdot k \cdot 2^k \cdot \bigg(1-\dfrac{1}{k \cdot 2^k}\bigg)^\ell.
\]
\begin{itemize}
\item[(a)] There are at most $(k \cdot 2^k)^n \cdot \delta_{k, n}(\ell)$ graphs in $\RainGraph[k](n)$ which are not $\ell$-good (for $k$). 

\item[(b)] There are at least $(k \cdot 2^k)^{n}\cdot \dfrac{1 - \delta_{k, n}(\ell)}{(k \cdot 2^k)^{k + 2^k} \cdot k!}$ graphs in $\RainGraph[k](n)$ which are $\ell$-good (for $k$).
\end{itemize}

\end{lemma}
\begin{proof}
For $r \leq \floor{\frac{n}{\ell}} - 3$ and $s \in [k] \times \Powerset([k])$ let 
\[
V_{r, s} = \{S \in \RainSeq[k](n) \st S^{-1}(s)\cap \big[(r+1) \cdot \ell, (r+2) \cdot \ell\big) = \emptyset\}
\]

Let $W_r = \bigcup \{V_{r, s} \st s \in [k] \times \Powerset([k])\}$. And let $NG_{\ell} = \bigcup \{W_r \st r \leq \floor{\frac{n}{\ell}} - 3\}$. Note $NG_\ell$ is the collection of non-$\ell$-good sequences. 

Let $\cU = (a_{\cU}, e_{\cU})$ be a $\RainSeq[k](n)$-valued random variable with uniform distribution. We then have 
\[
\Pr(\cU \in V_{r, s}) = \bigg(\dfrac{k \cdot 2^k-1}{k \cdot 2^k}\bigg)^\ell.
\]
Therefore 
\[
\Pr(\cU \in W_r) \leq k \cdot 2^k \cdot \bigg(\dfrac{k \cdot 2^k-1}{k \cdot 2^k}\bigg)^\ell. 
\]
and
\begin{align*}
\Pr(\cU \in NG_\ell) & \leq \bbfloor{\frac{n}{\ell}} \cdot k \cdot 2^k \cdot \bigg(\dfrac{k \cdot 2^k-1}{k \cdot 2^k}\bigg)^\ell\\
& = \bbfloor{\frac{n}{\ell}}\cdot k \cdot 2^k \cdot \bigg(1-\dfrac{1}{k \cdot 2^k}\bigg)^\ell
\end{align*}
So we have 
\[
|NG_\ell| \leq (k \cdot 2^k)^n \cdot \bbfloor{\frac{n}{\ell}}k \cdot 2^k \cdot \bigg(\dfrac{k \cdot 2^k-1}{k \cdot 2^k}\bigg)^\ell.
\]
Hence (a) holds. 

Let $S^- \in \AllColors<k>[k](k)$ and $S^+ \in \AllSetsColors<k>[{[n-2^k, n)}]({[n-2^k,n)})$. Let 
\[
A = \{S \in \RainSeq[k](n) \st S^-, S^+ \subseteq S\}.
\]
By \cref{All colors on the left imply you can determine equivalence of subsets on the right} and \cref{All sets of colors on the right imply you can determine equivalence of colors on the left} $|\SeqToGraph``[A]| = \frac{(k \cdot 2^k)^{n - k - 2^k}}{k!}$. Note the events $\cU \in A$ and $\cU \in NG_{\ell}$ are independent. Therefore we have 
\[
\Pr(\cU \in NG_{\ell} \given \cU \in A)  %
\leq \bbfloor{\frac{n}{\ell}}\cdot k \cdot 2^k \cdot \bigg(\dfrac{k \cdot 2^k-1}{k \cdot 2^k}\bigg)^\ell. 
\]
But this implies there are at least 
\begin{align*}
|\SeqToGraph``[A]| \cdot (1 - \delta_{k, n}(\ell))
& = \frac{(k \cdot 2^k)^{n - k - 2^k}}{k!} \cdot (1 - \delta_{k, n}(\ell)) \\
& =  (k \cdot 2^k)^{n}\cdot \dfrac{(1 - \delta_{k, n}(\ell))}{(k \cdot 2^k)^{k + 2^k} \cdot k!} 
\end{align*}
many $\ell$-good graphs in $\RainGraph[k](n)$. Hence (b) holds. 
\end{proof}

From \cref{Bounds on ell-good graphs} we therefore have the following bound on the proportion of $k$-rainbow threshold graphs which are not $\ell$-good. 

\begin{corollary}
\label{Bound on fraction of non-ell good graphs}
Suppose $k, n, \ell \in \w$. Let 
\[
\delta_{k, n}(\ell) = \bbfloor{\frac{n}{\ell}}\cdot k \cdot 2^k \cdot \bigg(1-\dfrac{1}{k \cdot 2^k}\bigg)^\ell.
\]
Then 
\[
\dfrac{|\{G \in \RainGraph[k](n)\st G \text{ is not }\ell\text{-good}\}|}{|\RainGraph[k](n)|} \leq \dfrac{\delta_{k, n}(\ell)}{1-\delta_{k, n}(\ell)} \cdot (k \cdot 2^k)^{k + 2^k} \cdot k!.
\]
\end{corollary}
\begin{proof}
This follows immediately from \cref{Bounds on ell-good graphs} and the fact that there are at least as many $k$-rainbow threshold graphs as there are $\ell$-good $k$-rainbow threshold graphs. 
\end{proof}

In particular any $\epsilon$ we can find a sufficiently large $\ell_\epsilon$ such that there the fraction of non $\ell_\epsilon$-good graphs in $\RainGraph[k](n)$ is $< \epsilon$ whenever $n \geq \ell_\epsilon^2$.

\begin{proposition}
\label{Good k-rainbow graphs are not isomorphic to k-1-rainbow graphs}
Suppose $\floor{\frac{n}{\ell}} \geq (k+1) \cdot 2^{k+1}(k+1 + 2^{k+1})$. No $\ell$-good $k+1$-rainbow threshold graph is isomorphic to a $k$-rainbow threshold graph. 
\end{proposition}
\begin{proof}
Suppose $S \in \RainSeq[k+1](n)$ is an $\ell$-good $k+1$-rainbow threshold sequence and $\cG =\SeqToGraph(S)$. We must show that $\cG$ is not isomorphic to any $k$-rainbow threshold graph. 

By \cref{Upper bound on neighborhood equivalence classes} it suffices to find an $X \subseteq [n]$ such that $\Neigh<\cG>[{[n]\setminus X}](X)$ has more than $k \cdot 2^{k} \cdot (1 + \frac{|X|}{2})$-many equivalence classes. 

For $r \leq \floor{\frac{n}{\ell}} - 3$ let $I_r = [(r+1) \cdot \ell, (r+2) \cdot \ell)$.  Let $X \subseteq [n]$ be such that the following hold where $(x_i)_{i \in [t]}$ is an increasing enumeration of $X$. 

\begin{itemize}
\item[(a)] For $i \leq [k+1]$, $x_i \in I_0$. 

\item[(b)] For $i \leq [2^{k+1}]$, $x_{t-i} \in I_{n-3}$. 

\item[(c)] For $i \in [k+1, t - 2^{k+1})$, $x_i \in I_{2(i-k)}$. 

\item[(d)] $S \in \AllColors<k+1>[n](\{x_i\}_{i \in [k]})$ and $S \in \AllSetsColors<k+1>[n](\{x_i\}_{i \in [t-2^{k+1}, t)})$. 

\item[(e)] $\floor{\frac{n}{\ell}} + k-4 + 2^{k+1} \geq t \geq (k+1) \cdot 2^{k+1}(k+1 + 2^{k+1})$. 
\end{itemize}

Note such a sequence exists as $\floor{\frac{n}{\ell}} \geq (k+1) \cdot 2^{k+1}(k+1 + 2^{k+1})$ and $S$ is $\ell$-good. 

Let $Y = [x_{k+1}+1, x_{t - 2^k})$. For $j \in [t - k - 2^k]$ let 
\[
X_j = \{x_i\}_{i \in [k +1  + j]} \cup \{x_i\}_{i \in [t-2^{k+1}, t)}
\]

By \cref{All colors on the left imply you can determine equivalence of subsets on the right} and \cref{All sets of colors on the right imply you can determine equivalence of colors on the left} and conditions (a), (b) and (d) we hae that whenever $y_0, y_1 \in Y$ with $S(y_0) \neq S(y_1)$ then $\neg (y_0 \Neigh<S>[Y \setminus X](X_0) y_1)$. 

Therefore the $\Neigh<S>[Y \setminus X](X_0)$-neighborhood of an element determines, up to $\sim$, the value of $S$ on its elements. 

For $i \in [t - 2^k - k]$ let $r_i$ be the number of $\Neigh<S>[Y \setminus X](X_i)$-equivalence classes. 

Condition (c) ensures that for $i \in [k+1, t - 2^{k+1}-1)$ we have
\[
(\forall s \in [k+1] \times [2^{k+1}])\, S^{-1}(s) \cap [x_i + 1, x_{i+1}) \neq \emptyset. 
\]
Therefore, for $i \in [t- 2^{k+1} - k-2]$ we have $r_{i+1} - r_i = \frac{1}{2} \cdot (k+1) \cdot 2^{k+1}$. 
In particular this means that $\Neigh<S>[Y \setminus X](X)$ has $(k+1) \cdot 2^{k+1} + \frac{1}{2} \cdot (t - k -1- 2^{k+1}) \cdot (k+1) \cdot 2^{k+1}$-many equivalence classes. 

Hence $\Neigh<S>[{[n] \setminus X}](X)$ has at least $(k+1) \cdot 2^{k+1} + \frac{1}{2} \cdot (t - k -1- 2^{k+1}) \cdot (k+1) \cdot 2^{k+1}$-many equivalence classes. 
But 
\begin{align*}
[k+1 \cdot 2^{k+1}  & + \tfrac{1}{2} \cdot (t - k-1 - 2^{k+1}) \cdot (k+1) \cdot 2^{k+1}]  - [k \cdot 2^{k}\cdot (1 + \tfrac{1}{2} \cdot t)] \\
& = ((k+1) \cdot 2^{k+1} - k \cdot 2^{k}) \cdot (1 + \tfrac{1}{2} t) - (k+1) \cdot 2^{k}(k+1 + 2^{k+1})\\
& > (1 + \tfrac{1}{2} \cdot t) - (k+1) \cdot 2^{k}(k+1 + 2^{k+1}) \geq 0.
\end{align*}
Therefore $\Neigh<\cG>[{[n]\setminus X}](X)$ has more than $k \cdot 2^{k} \cdot (1 + \frac{|X|}{2})$-many equivalence classes as desired. 
\end{proof}

\begin{theorem}
Suppose $k, n \in \w$ where 
\[
\floor{\frac{n}{\ell}} \geq 2^{3k+3}
\]
and 
\[
n \geq \dfrac{-2^{4k+7}}{\log_2\bigg(1-\dfrac{1}{(k+1) \cdot 2^{k+1}}\bigg)}
\]
Then 
\[
\dfrac{ \{\cG \in \RainGraph[k+1](n) \st (\exists \cH \in \RainGraph[k](n))\, \cG \cong \cH\}}{\RainGraph[k+1](n)} \leq \bigg(\bigg(1-\dfrac{1}{(k+1) \cdot 2^{k+1}}\bigg)^\frac{1}{2^{3k+3}}\bigg)^n \cdot 2^{2^{3k+5}}
\]
\end{theorem}
\begin{proof}
By \cref{Good k-rainbow graphs are not isomorphic to k-1-rainbow graphs} it suffices to show that there is an $\ell$ such that $\floor{\frac{n}{\ell}} \geq 2^{3k+3}\geq (k+1) \cdot 2^{k+1}\cdot (k + 1 + 2^{k+2})$ and 
\[
\dfrac{ \{\cG \in \RainGraph[k+1](n) \st \cG\text{ is not }\ell\text{-good}\}}{\RainGraph[k+!](n)} \leq \bigg(\bigg(1-\dfrac{1}{(k+1) \cdot 2^{k+1}}\bigg)^\frac{1}{2^{3k+3}}\bigg)^n \cdot 2^{2^{3k+5}}
\]

But by \cref{Bound on fraction of non-ell good graphs} it suffices to find an $\ell$ such that 
\[
\dfrac{\delta}{1-\delta} \cdot ((k+1) \cdot 2^{k+1})^{k+1 + 2^{k+1}} \cdot (k+1)! \leq \epsilon
\]
where $\delta = \floor{\frac{n}{\ell}}\cdot (k+1) \cdot 2^{k+1} \cdot \bigg(1-\dfrac{1}{(k+1) \cdot 2^{k+1}}\bigg)^\ell$. 

Let $\ell$ be such that $2^{3k+3} \cdot \ell \leq n < (2^{3k+3} +1) \cdot \ell$. Then $\floor{\frac{n}{\ell}} = 2^{3k+3}$. We then have 

\begin{align*}
\delta & = 2^{3k+3}\cdot (k+1) \cdot 2^{k+1} \cdot \bigg(1-\dfrac{1}{(k+1) \cdot 2^{k+1}}\bigg)^\ell \\
& \leq 2^{5k+4} \cdot \bigg(1-\dfrac{1}{(k+1) \cdot 2^{k+1}}\bigg)^{\frac{n}{2^{3k+3} +1}} \\
& \leq 2^{5k+4} \cdot \bigg(1-\dfrac{1}{(k+1) \cdot 2^{k+1}}\bigg)^{\frac{n}{2^{3k+4}}}
\end{align*}

Therefore 
\begin{align*}
\log_2(\delta) & \leq 5k+4 + \frac{n}{2^{3k+4}} \cdot \log_2\bigg(1-\dfrac{1}{(k+1) \cdot 2^{k+1}}\bigg) \\
& \leq 5k+4 + \frac{-2^{4k+7}}{\log_2\bigg(1-\dfrac{1}{(k+1) \cdot 2^{k+1}}\bigg)} \cdot  \frac{1}{2^{3k+4}} \cdot \log_2\bigg(1-\dfrac{1}{(k+1) \cdot 2^{k+1}}\bigg) \\
& \leq 5k+4  - 2^{4k+7} \cdot  \frac{1}{2^{3k+4}} \leq  5k+4  - 2^{k+3} < 0. 
\end{align*}
Therefore $\delta < 1$ and $\frac{\delta}{1- \delta} \leq 2 \cdot \delta$. 

Therefore
\begin{align*}
\dfrac{\delta}{1-\delta} & \cdot ((k+1) \cdot 2^{k+1})^{k+1 + 2^{k+1}} \cdot (k+1)!  \leq \delta \cdot 2^{2^{3k+4}} \\
& =  \bbfloor{\frac{n}{\ell}}\cdot (k+1) \cdot 2^{k+1} \cdot \bigg(1-\dfrac{1}{(k+1) \cdot 2^{k+1}}\bigg)^\ell \cdot 2^{2^{3k+4}} \\
& \leq \bigg(1-\dfrac{1}{(k+1) \cdot 2^{k+1}}\bigg)^\ell \cdot 2^{2^{3k+5}}\\
& \leq \bigg(1-\dfrac{1}{(k+1) \cdot 2^{k+1}}\bigg)^\frac{n}{2^{3k+3}} \cdot 2^{2^{3k+5}}  \\
& = \bigg(\bigg(1-\dfrac{1}{(k+1) \cdot 2^{k+1}}\bigg)^\frac{1}{2^{3k+3}}\bigg)^n \cdot 2^{2^{3k+5}}
\end{align*}
as desired. 
\end{proof}
In particular we have the following corollary. 
\begin{corollary}
Suppose $k, n \in \w$.
Then 
\[
\lim_{n \to \infty}\dfrac{ \{\cG \in \RainGraph[k](n) \st (\exists \cH \in \RainGraph[k-1](n))\, \cG \cong \cH\}}{\RainGraph[k](n)}  = 0.
\]
\end{corollary}

\section{$0$-$1$ Law}

We now end by showing that for no $k$ do the collection of $k$-rainbow threshold graphs satisfy a $0$-$1$ law. 

\begin{theorem}
There is a single sentence $\varphi$ of first order logic in the language of graphs such that for all $k \geq 1$
\[
\lim_{n \to \infty} \dfrac{|\{\cG \in \RainGraph[k](n) \st \cG \models \varphi\}|}{|\RainGraph[k](n)|} \not \in \{0, 1\}.
\]
\end{theorem}
\begin{proof}
Let $\varphi$ be the sentence $(\exists x)(\forall y)\, \neg E(x, y)$, i.e. that there is an isolated vertex. Let $\psi$ be the sentence $(\exists x)(\forall y)\,x \neq y \rightarrow E(x, y)$. Note we cannot have both $\varphi$ and $\psi$ holding in any graph of size at least $2$. Fix $k \geq 1$.

Let $S^- \in \AllColors<k>[k](k)$ and $S^+ \in \AllSetsColors<k>[{[n-2^k, n)}]({[n-2^k,n)})$. Let 
\[
A = \{S \in \RainSeq[k](n) \st S^-, S^+ \subseteq S\}.
\]
By \cref{All colors on the left imply you can determine equivalence of subsets on the right} and \cref{All sets of colors on the right imply you can determine equivalence of colors on the left} $|\SeqToGraph``[A]| = \frac{(k \cdot 2^k)^{n - k - 2^k}}{k!}$

For $s \in \Powerset([k])$ let $B_s = \{(a, e) \in \RainSeq[k](n+1) \st (a, e)\rest[n] \in A\text{ and }e(n) = s\}$. 

Note $|B_\emptyset| = |B_{[k[}| = |\SeqToGraph``[A]| = \frac{(k \cdot 2^k)^{n - k - 2^k}}{k!} = \frac{(k \cdot 2^k)^{n - k - 2^k}}{k!}$. 

Therefore we have 
\begin{align*}
\dfrac{|\{\cG \in \RainGraph[k](n+1) \st \cG \models \varphi\}|}{|\RainGraph[k](n+1)|} & \geq \dfrac{(k \cdot 2^k)^{n - k - 2^k}}{k!\cdot (k \cdot 2^k)^n} \\
& = \dfrac{1}{k!\cdot (k \cdot 2^k)^{k + 2^k}} \geq 0
\end{align*}
and 
\begin{align*}
\dfrac{|\{\cG \in \RainGraph[k](n+1) \st \cG \models \psi\}|}{|\RainGraph[k](n+1)|} & \geq \dfrac{(k \cdot 2^k)^{n - k - 2^k}}{k!\cdot (k \cdot 2^k)^n} \\
& = \dfrac{1}{k!\cdot (k \cdot 2^k)^{k + 2^k}} \geq 0
\end{align*}
Therefore 
\[
\lim_{n \to \infty} \dfrac{|\{\cG \in \RainGraph[k](n) \st \cG \models \varphi\}|}{|\RainGraph[k](n)|} \not \in \{0, 1\}
\]
as desired. 
\end{proof}

\bibliography{bibliography}

\begin{thebibliography}{1}

\bibitem{MR1686154}
A.~Brandst\"{a}dt, V.~B. Le, and J.~P. Spinrad.
\newblock {\em Graph classes: a survey}.
\newblock SIAM Monographs on Discrete Mathematics and Applications. Society for
  Industrial and Applied Mathematics (SIAM), Philadelphia, PA, 1999.

\bibitem{MR479384}
V.~Chv\'{a}tal and P.~L. Hammer.
\newblock Aggregation of inequalities in integer programming.
\newblock In {\em Studies in integer programming ({P}roc. {W}orkshop, {B}onn,
  1975)}, Ann. Discrete Math., Vol. 1, pages 145--162. North-Holland,
  Amsterdam-New York-Oxford, 1977.

\bibitem{MR2573956}
P.~Diaconis, S.~Holmes, and S.~Janson.
\newblock Threshold graph limits and random threshold graphs.
\newblock {\em Internet Math.}, 5(3):267--320 (2009), 2008.

\bibitem{MR488948}
P.~B. Henderson and Y.~Zalcstein.
\newblock A graph-theoretic characterization of the {${\rm PV}_{{\rm chunk}}$}
  class of synchronizing primitives.
\newblock {\em SIAM J. Comput.}, 6(1):88--108, 1977.

\bibitem{MR1417258}
N.~V.~R. Mahadev and U.~N. Peled.
\newblock {\em Threshold graphs and related topics}, volume~56 of {\em Annals
  of Discrete Mathematics}.
\newblock North-Holland Publishing Co., Amsterdam, 1995.

\end{thebibliography}
 \bibliographystyle{abbrv}

\end{document}